\documentclass[12pt,a4paper]{article}
\usepackage[OT1]{fontenc}
\usepackage[latin1]{inputenc}
\usepackage[english]{babel}
\usepackage{amsfonts}
\usepackage{amsthm}
\newtheorem{thm}{Theorem}[section]
\newtheorem{prop}[thm]{Proposition}
\newtheorem{lem}[thm]{Lemma}
\newtheorem{cor}[thm]{Corollary}

\theoremstyle{definition}
\newtheorem{rem}[thm]{Remark}
\newtheorem{exmp}[thm]{Example}

\def\bh{{\cal B}({\cal H})}

\def\bp{{\cal B}_p({\cal H})}

\begin{document}

\title{\vspace*{0cm}Norm Inequalities in Operator Ideals}

\date{}
\author{Gabriel Larotonda\footnote{Partially supported by IAM-CONICET.}}

\maketitle

\abstract{\noindent In this paper we introduce a new technique for proving norm inequalities in operator ideals with an unitarily invariant norm. Among the well known inequalities which can be proved with this technique are the L\"owner-Heinz inequality, inequalities relating various operator means and the Corach-Porta-Recht inequality. We prove two general inequalities and from them we derive several inequalities by specialization, many of them new. We also show how some inequalities, known to be valid for matrices or bounded operators, can be extended with this technique to normed ideals in $C^*$-algebras, in particular to the noncommutative $L^p$-spaces of a semi-finite von Neumann algebra.
\footnotesize{\noindent }}

\setlength{\parindent}{0cm} 

\section{Introduction}

Let $\bh$ denote the set of bounded linear operators acting on a separable and complex Hilbert space ${\cal H}$. Among the basic properties of the usual supremum norm of $\bh$ stands its \textit{unitarily invariance}, namely
$$
\|UXV\|=\|X\|
$$
for $X\in \bh$ and $U,V$ unitary operators of $\bh$. Related to it is the elementary inequality
$$
\|Z\pm i X Z\|\ge \|Z\|,
$$
valid for self-adjoint $X\in\bh$ and any $Z\in \bh$, which can be restated using the left multiplication operator $L_X$,
$$
\|(1\pm i L_X) Z\|\ge \|Z\|.
$$
This property combined with the Weierstrass factorization theorem for entire functions, which allows us to use the inequality above repeatedly, turns out to be a quite powerful tool to prove nontrivial norm inequalities. This fact was pointed out in \cite{neeb} by Neeb for the particular function $f(z)=z^{-1}\sin(z)$, so we would like to stress that the main idea in Theorem \ref{in} below stems from Neeb's paper. With this approach we prove new inequalities, and we show how other inequalities, that through the years have been proved by a long list of authors with a variety of techniques, can be derived by specialization and further generalized to a broader setting. 

\medskip

This is not the place to give a thorough discussion on the subject of norm inequalities in spaces of operators, but let us just mention that it has deep connections with the theory of majorization and Schur products for matrices \cite{man}, and that recently Hiai and Kosaki \cite{kosakih}  developed a striking technique that involves Fourier transforms to prove norm inequalities for bounded operators.

\medskip

We have chosen to state our results in the setting of ideals of operators with an unitarily invariant norm in a $C^*$-algebra, where computations require a slightly more delicate approach due to the different spectra that may arise when changing the norm. In our discussion are then included the noncommutative $L^p$-spaces $L^p(M,\tau)$ of a semi-finite von Neumann algebra $M$ with trace $\tau$, which are the completion of the ideals
$$
{\cal I}^{\tau,p}=\{x\in M: \tau(\mid x\mid^p)<+\infty\}
$$
relative to the unitarily invariant norm given by $\|x\|_p=\tau(\mid x\mid^p)^{\frac1p}$, where as usual $p\in [1,+\infty]$ and $L^{\infty}(M,\tau)=M$.

\medskip

The main new results of this paper are Theorems \ref{in}, \ref{ineq} and \ref{comparo}, and the paper is organized as follows: In Section \ref{secmain} we recall some elementary notions on dissipative operators in order to prove an inequality involving  perturbations of the identity. Then these building blocks are used to prove our first main inequality related to entire functions of finite order (with purely imaginary roots) by means of Weierstrass' factorization theorem. We are mainly concerned with the left and right multiplication representations in a $C^*$-algebra $M$, but the technique generalizes to a broader setting (see Theorem \ref{in} and the concluding remarks of the paper). In Section \ref{seccpr} we prove a generalization of the Corach-Porta-Recht inequality and some related inequalities. In Section \ref{secmero} we prove our second main inequality concerning meromorphic functions with interlaced (purely imaginary) roots, and from there we derive several results by specialization. Among them are the exponential metric increasing property, the comparison among various means of operators, the Lipschitz continuity of the absolute value map and the C\"ordes (also known as the L\"owner-Heinz) inequality. Finally, section \ref{secfinal} contains some immediate generalizations of the main results of this paper to other contexts.

\medskip

A similar technique (involving the Weierstrass factorization theorem) has been used by Kosaki \cite{kosakinew}, to prove that certain ratios of real analytic (scalar) functions are positive definite (in fact, infinitely divisible, see Theorem 2 and Corollary 3 of \cite{kosakinew}), thus obtaining norm inequalities for the corresponding operator means. The approach in \cite{kosakinew} allows to detect failure of norm inequalities by computing Fourier transforms and showing that the corresponding scalar function is not positive definite. In our setting, on the other hand, it is possible to deal with inequalities with complex parameters since we do not require the relevant entire functions to be real valued on the real axis (see Remark \ref{compl1} and Example \ref{compl2}).

\section{First main inequality}\label{secmain}

Let $E$ be a (real or complex) Banach space. Let ${\cal B}(E)$ stand for the Banach algebra of linear bounded operators in $E$. Let $GL(E)$ stand for the group of invertible elements of ${\cal B}(E)$, which is a Banach-Lie group, open in ${\cal B}(E)$. Let ${\mathfrak ul}(E)$ be the group of linear isometries of $E$, that is
$$
{\mathfrak ul}(E)=\{g\in GL(E): \|g\|=\|g^{-1}\|\le 1\}.
$$
Here $\|\cdot\|$ stands for the usual supremum norm. Then ${\mathfrak ul}(E)$ is a real Banach-Lie group with a topology which is possibly finer than the norm topology of $GL(E)$, and 
$$
{\mathfrak u}=\{T\in {\cal B}(E):\|e^{sT}\|\le 1 : \,\forall s\in \mathbb R\}
$$
is the Banach-Lie algebra of ${\mathfrak ul}(E)$. Recall that an operator $A\in {\cal B}(E)$ is \textit{expansive} if
$$
\|Az\|_E\ge \|z\|_E
$$
for any $z\in E$. For $z\in E$, let $D(z)\subset E^*$ stand for the set of norming functionals of $z$,
$$
D(z)=\{\phi\in E^*:\|\phi\|=1\;\mbox{ and }\; \phi(z)=\|z\|_E\}.
$$
An operator $T\in {\cal B}(E)$ is \textit{dissipative} if, for any $z\in E$,
$$
Re\,\phi(Tz)\le 0 \quad\mbox{ for any }\; \phi\in D(z).
$$
Our main reference on the subject of dissipative operators (sometimes called \textit{accretive} in the literature) is the book by Hille and Phillips \cite{hplibro}. The proof of the following theorem can be found in \cite[Theorem II.2]{neeb}.

\begin{thm}\label{disi}
If $T\in {\cal B}(E)$, the following assertions are equivalent.
\begin{enumerate}
\item $T$ is dissipative.
\item For each $z\in E$, $Re\,\phi(Tz)\le 0$ for some $\phi\in D(z)$.
\item $\|e^{sT}\|\le 1$ for any $s\ge 0$.
\item $1-sT$ is expansive and invertible for any $s\ge 0$
\end{enumerate}
\end{thm}
In particular if $T$ is dissipative, $1-sT$ is invertible and the inverse is a bounded contraction. We prove here two lemmas which will be useful later.

\begin{lem}
If $A\in {\cal B}(E)$ is a contraction (i.e. $\|A\|\le 1$) then $A-1$ is a dissipative operator. In particular
$$
(1-s(A-1))^{-1}=(1+s-sA)^{-1}
$$
is also a contraction for any $s\ge 0$.
\end{lem}
\begin{proof}
Let $z\in E$, $\phi\in D(z)$. Then
$$
Re\,\phi((A-1)z)=Re\,\phi(Az)-\phi(z)\le \mid \phi(Az)\mid -\|z\|_E\le \|Az\|_E-\|z\|_E\le 0.
$$
\end{proof}

\begin{lem}\label{divido}
Let $a,b\in \mathbb R$, $a/b>1$. Let $A\in {\cal B}(E)$ be a contraction. Then
$$
T= \frac{a}{b}+\left(1-\frac{a}{b}\right)A
$$
is invertible and the inverse is a contraction.
\end{lem}
\begin{proof}
Put $s=\frac{a}{b}-1>0$. Then the result follows from the previous lemma.
\end{proof}

When $E$  is complex, the set $Herm(E)=i{\mathfrak u}$ is the set of \textit{Hermitian} elements of ${\cal B}(E)$. If $T\in Herm(E)$, then its norm can be computed using the spectral radius formula \cite[Chapter 4]{upmeier}
$$
\|T\|=\sup\{ |\lambda|: \lambda\in\sigma(T)\}.
$$
Note that when $E$ is a complex Banach space, 
$$
Herm(E)=Diss(E)\cap (-Diss(E)),
$$
where $Diss(E)$ denotes the cone of dissipative operators.

Let $A$ be a complex involutive algebra, and $A_h$  stand for the set of self-adjoint elements of $A$,
$$
A_h=\{x\in A:x^*=x\}.
$$
Assume that $A$ acts on a (real or complex) Banach space $(E,\|\cdot\|_E)$, so that $\pi(a)\in {\cal B}(E)$ for any $a\in A$. Assume further that $E$ has a unitarily invariant norm in the sense that 
\begin{equation}\label{uni}
\|e^{\pi(i\,x)}z\|_E=\|z\|_E 
\end{equation}
for any $z\in E$ and any $x\in A_h$. Then if $s\in \mathbb R$, replacing $x$ by $s\,x$ in  (\ref{uni}) yields
$$
\|e^{s\,\pi(i\,x)}z\|_E=\|z\|_E
$$
for any $x\in A_h$ and any $z\in E$, so $\pi(x)\in {\cal B}(E)$ is Hermitian. Then
$$
\|z+ s\, \pi(ix) z\|_E\ge \|z\|_E
$$
for any $z\in E$ and any $s\in\mathbb R$. 

\begin{rem}\label{entera}
Let $F:\mathbb C\to \mathbb C$ be an entire function and 
$$
M(r)=\max\limits_{|z|\le r}|F(z)|=\max\limits_{|z|= r}|F(z)|.
$$
The \textit{order of growth} $\rho$ of an entire function is given by
$$
\rho=\limsup\limits_{r\to +\infty}\frac{\ln\ln M(r)}{\ln r}.
$$
\end{rem}
It is easy to check that if $F,G$ are entire functions of order $\rho_F,\rho_G$ respectively then
$$
\rho_{F+G},\rho_{FG}\le \max\{\rho_F,\rho_F\}.
$$
Our main reference on entire functions is the monograph \cite{levin} by B. Levin. Any function $F$ of finite order $\rho\le 1$ can be written in its Weierstrass expansion
$$
F(z)=z^{j} e^{\alpha z}\lambda \prod\limits_{k\in \mathbb Z}\left(1-\frac{z}{z_k} \right)e^{\frac{z}{z_k}}.
$$
Here $\{z_k\}\subset \mathbb C-\{0\}$ are the nonzero roots of $F$ and $\lambda=F(z)/z^j\left|_{z=0}\right.\in\mathbb C$, where $j\in \mathbb N$ is the order of zero as a root of $F$. The product converges uniformly to $F$ on compact sets of $\mathbb C$. Note that the exponent $\alpha\in\mathbb C$ can be computed via
$$
\alpha= \frac{d}{dz}\ln(F(z)/z^j)\left|_{z=0}\right. .
$$

\begin{thm}\label{in}
Let $A$ be an involutive complex algebra which acts on a complex Banach space $E$. Assume that the norm of $E$ is unitarily invariant,
$$
\|e^{\pi(ix)}z\|_E=\|z\|_E
$$
for any $z\in E$ and any $x\in A_h$. Let $F:\mathbb C\to \mathbb C$ be an entire function of order $\rho\le 1$, such that the nonzero roots $\{w_k=iz_k\}\subset \mathbb C-\{0\}$ of $F$ are purely imaginary, namely
$$
F(z)=\lambda\, e^{\alpha z} z^n\prod \left(1-\frac{iz}{z_k} \right) e^{\frac{iz}{z_k}}
$$
is the Weierstrass expansion of $F$, with $\lambda=F(z)/z^n\left|_{z=0}\right.\in\mathbb C$ and $z_k\in \mathbb R-\{0\}$. Then
\begin{enumerate}
\item If $x\in A_h$ and $z\in E$ then
$$
\|F(\pi(x))z\|_E\ge\, \mid \lambda\mid \|\; e^{\alpha \pi(x)} \pi(x)^n(z)\|_E.
$$
\item If $\sigma_{{\cal B}(E)}(\pi(x))\cap \{w_k\}=\emptyset$, then $F(\pi(x))$ is invertible in ${\cal B}(E)$.
\item In particular, if $\mid F(0)\mid=1$ and $\alpha$ is purely imaginary, then
$$
\|F(\pi(x))z\|_E\ge \|z\|_E.
$$
\end{enumerate}
\end{thm}
\begin{proof}
Clearly $F(\pi(x))$ is well defined by the analytic functional calculus,  and then the conclusion follows from the remarks above since
$$
F(\pi(x))=\lambda\, e^{\alpha \pi(x)}\pi(x)^n \lim_N \prod_{k=1}^N \left[1-\frac{1}{z_k}\pi(ix)\right]e^{\frac{\pi(ix)}{z_k}}.
$$
If the spectrum of $\pi(x)$ does not intersect the roots of $F$, then clearly $$
0\notin \sigma_{{\cal B}(E)}F(\pi(x)).
$$
\end{proof}

\begin{rem}
The typical setting in operator algebras occurs when either $A=E=\bh$, where ${\cal H}$ is a complex Hilbert space and $\pi=ad$, the adjoint representation, or  else when $E$ is one of the $p$-Schatten \cite{simon} ideals of compact operators in $\bh$. We wish to extend this result to normed ideals of any $C^*$-algebra, the obvious examples being the noncomutative $L^p$  spaces of Murray-von Neumann and Segal \cite{segal,nelson} arising from a semi-finite trace in a semi-finite factor.
\end{rem}

\medskip

\begin{rem}\label{espectro}
Let us fix the notation for the next theorem. Let $M$ be a $C^*$-algebra and $\sigma(X)$ stand for the spectrum of $X$ relative to $M$. Let ${\cal I}_0$ be a normed ideal in $M$ with norm $\|\cdot\|_{\cal I}$ such that
\begin{equation}\label{ideal}
\|XYZ\|_{\cal I}\le \|X\|\; \|Y\|_{\cal I}\, \|Z\|
\end{equation}
for $X,Z\in M$ and $Y\in {\cal I}_0$. Note that $\|UXV\|_{\cal I}=\|X\|_{\cal I}$ for unitary $U,V\in M$ and  $X\in {\cal I}_0$. Let ${\cal I}$ stand for the completion of the linear space ${\cal I}_0$ relative to the norm $\|\cdot\|_{\cal I}$. We say that ${\cal I}$ is a \textit{normed ideal in $M$ with an unitarily invariant norm}.

Let $L$ and $R$ stand for the left and right multiplication representations, that is $L_X(T)=XT$, $R_X(T)=TX$. Then (since $L_X$ and $R_Y$ commute for $X,Y\in M$)
$$
e^{L_X+R_Y}T=e^X T e^{Y}.
$$
By inequality (\ref{ideal}), the maps $L_X$ and $R_Y$ extend to bounded linear operators
$$
\widetilde{L_X},\widetilde{R_X}\in {\cal B}({\cal I}).
$$
Note that 
\begin{equation}\label{ad}
e^{\widetilde{L_X}+\widetilde{R_Y}}=\widetilde{L_{e^X}}\widetilde{R_{e^{Y}}}.
\end{equation}
Let $L(M)$ stand for the left multiplication representation of $M$, which is a closed subalgebra of ${\cal B}(M)$. Then $M$ is isomorphic to $L(M)$, and $\sigma_{L(M)}(L_X)=\sigma(X)$. Since $L_X\mapsto \widetilde{L_X}$ is an injective homomorphism of complex unital Banach algebras which gives the inclusion $L(M)\hookrightarrow {\cal B}({\cal I})$, then $\sigma_{B({\cal I})}(\widetilde{L_X})\subset \sigma_{L(M)}(L_X)=\sigma(X)$. The same remark holds for $R(M)$, hence
$$
\sigma_{{\cal B}(\cal I)}(\widetilde{L_X}+\widetilde{R_Y})\subset \sigma_{{\cal B}(\cal I)}(\widetilde{L_X})+\sigma_{{\cal B}(\cal I)}(\widetilde{R_X})\subset \sigma(X)+\sigma(Y).
$$
In particular, if $X,Y\in M_h$, then 
$$
\sigma_{{\cal B}(\cal I)}(\widetilde{L_X}+\widetilde{R_Y})\subset \mathbb R.
$$
\end{rem}
We shall omit the tilde from now on, and write $e^XT$ instead of $\widetilde{L_{e^X}}T$, etc.

\begin{lem}
If $X,Y\in M_h$ and  ${\cal I}$  is a complex normed ideal in $M$ with a unitarily invariant norm, then
\begin{enumerate}
\item $\sigma_{{\cal B}(\cal I)}(  {L_X}+{R_Y}    )\subset \mathbb R$
\item $\|e^{i\, (L_X+R_Y)}T\|_{\cal I}=\|T\|_{\cal I}$ for any $T\in {\cal I}$
\item $1 \pm \frac{i}{r}(L_X+R_Y)$ is expansive and invertible for each $r\ne 0$.
\end{enumerate}
\end{lem}
\begin{proof}
The first assertion is the above remark. The second assertion combines (\ref{ad}) and the unitarily invariance of the norm. The third assertion follows from Theorem \ref{disi}.
\end{proof}

\begin{thm}\label{ineq}
Let $M$ be a $C^*$ algebra and  $({\cal I},\|\cdot\|_{\cal I})$ a complex normed ideal in $M$ with an unitarily invariant norm. Let $F:\mathbb C\to \mathbb C$ be an entire function of order $\rho\le 1$ such that the nonzero roots of $F$ are purely imaginary, namely
$$
F(z)=\lambda\, e^{\alpha z} z^n\prod \left(1-\frac{iz}{z_k} \right) e^{\frac{iz}{z_k}}
$$
is the Weierstrass factorization of $F$, with $\lambda=F(z)/z^n\left|_{z=0}\right.\in\mathbb C$ and $z_k\in\mathbb R-\{0\}$. Then
\begin{enumerate}
\item If $X,Y\in M_h$ and $T\in {\cal I}$, 
$$
\|F(L_X+R_Y)T\|_{\cal I}\ge\, \mid \lambda\mid \left\| \sum_{k=1}^n  
\left(\begin{array}{c}
	n \\ k
\end{array}\right)
X^k e^{b X} T e^{bY } Y^{n-k}\right\|_{\cal I},
$$
with $b=Re(\alpha)$.
\item If $F(0)\ne 0$, then $F_{X,Y}=F(L_X+R_Y)$ is invertible in ${\cal B}({\cal I})$, with 
$$
\|F_{X,Y}^{-1}\|_{\cal B({\cal I})}\le |\lambda|^{-1} \|e^{-b X}\| \, \|e^{-b Y}\|.
$$
\item In particular, if $\mid F(0)\mid=1$ and $\alpha$ is purely imaginary, then
$$
\|F_{X,Y}T\|_{\cal I}\ge \|T\|_{\cal I}
$$
and the inverse of $F_{X,Y}$ is a contraction.
\item If ${\cal I}$ is a real ideal, the same assertions hold if we require further that $F(\mathbb R)\subset\mathbb R$.
\end{enumerate}
\end{thm}
\begin{proof}
As in Theorem \ref{in}, each nontrivial factor in the expansion of $F_{X,Y}$ is either a unitary operator or an expansive operator. To prove the second assertion, assume that $F_{X,Y}$ is not invertible. Then $0\in \sigma_{{\cal B}(\cal I)}(F_{X,Y})=F(\sigma_{{\cal B}(\cal I)}( L_X+R_Y ))$, namely there exists $t\in \sigma_{{\cal B}(\cal I)}(L_X+R_Y)$ such that $F(t)=0$. As remarked above, $\sigma_{{\cal B}(\cal I)}(L_X+R_Y)\subset \mathbb R$ and then $t=0$ since the roots of $F$ are purely imaginary, namely $F(0)=0$.
\end{proof}

\begin{exmp}
Let $\Gamma$ stand for the usual Gamma function, which has simple poles in the nonpositive integers. Then
$$
g(z)=\frac{1}{z\Gamma(z)}=e^{\gamma z} \prod_{n\ge 1} \left(1+\frac{z}{n}\right) e^{-z/n},
$$
where $\gamma$ is the Euler-Mascheroni constant. Since $g(0)=1$, if $H=iX\in i M_h$ and $T\in{\cal I}$ then $g(L_H)$ is invertible in ${\cal B}({\cal I})$, and the inverse is a contraction. Hence
$$
\|\Gamma(H)HT\|_{\cal I}\le \|T\|_{\cal I}
$$
for any skew-adjoint $H$ and any $T\in {\cal I}$.
\end{exmp}

\medskip

\begin{rem}\label{eledos}
Since the involution is isometric for the $p$-norms in ${\cal I}^{\tau,p}$, it extends to an isometry $J:L^p\to L^p$. Then
$$
L^p=L^p_h\oplus i L^p_{h}
$$ 
where $L^p_h=\{T\in L^p: JT=T\}$. When $p=2$, the space $L^2(M,\tau)$ is the \textit{standard Hilbert space} where $M$ is represented \textit{via} the left multiplication representation, with inner product
$$
<W,Z>_{\tau}=\tau(WJ(Z)).
$$
If $X\in M_h$, each $L_X$ or $R_X$ is a symmetric operator of $L^2(M,\tau)$ since
$$
<L_XW,Z>_{\tau}=\tau(XWJ(Z))=\tau(WJ(XZ))=<W,L_XZ>_{\tau},
$$
and it can be proven that they admit self-adjoint extensions \cite{nelson}. Then if $F$ is a continuous function that maps $\mathbb R$ into $\mathbb R$, the same is true for the operator $F(L_X+R_Y)$, i.e. 
$$
F(L_X+R_Y)\in {\cal B}(\, L^2(M,\tau)\,)_{h}.
$$
In particular its norm can be computed using the spectral radius formula and of course
$$
\sigma_{{\cal B}(\, L^2(M,\tau)\,)}( F(L_X+R_Y))=F( \sigma_{{\cal B}(\, L^2(M,\tau)\,)}(L_X+R_Y)).
$$
\end{rem}

\section{Applications}\label{seccpr}

In this section we indicate how to apply Theorem \ref{ineq} to derive inequalities related to the Corach-Porta-Recht inequality, in the setting of normed ideals in a $C^*$-algebra. 

\subsection{The generalized CPR inequality}

\begin{thm}\label{grande}
Let $T\in {\cal I}$, where ${\cal I}$ is a normed ideal with an unitarily invariant norm in a $C^*$-algebra $M$. Let $S,R\in M$ be positive and invertible, $S=e^X, R=e^{-Y}$ with $X,Y\in M_h$. Let $(r,\theta)\in [0,2\pi)\times [-2,2]$, and let
$$
\Psi_{S,R,\theta,r}(T)=e^{i\theta}rT+e^{i\theta}STR^{-1} +S^{-1}TR,
$$
$$
C(r,\theta)= (r+1)^2+2(r+1)\cos(\theta)+1
$$
and
$$
b(r,\theta)= r(1-\cos\theta)C(r,\theta)^{-2}.
$$
Then $ \Psi_{S,R,\theta,r}\in B({\cal I})$ is invertible and
$$
\| \Psi_{S,R,\theta,r}(T) \|_{\cal I}\ge C(r,\theta) \|S^{b} T R^{-b}\|_{\cal I},
$$

with the exceptions $(r,\theta)=(0,\pi)$ where
\begin{equation}
\|STR^{-1}-S^{-1}TR\|_{\cal I}\ge 2\| XT-TY\|_{\cal I} \label{senoo}
\end{equation}
and $(r,\theta)=(-2,0)$, where
$$
\|STR^{-1}+S^{-1}TR-2T\|_{\cal I}\ge 2\| XT-TY\|_{\cal I}.
$$
\end{thm}
\begin{proof}
Let $F(z)=e^{i\theta}(r+e^z)+e^{-z}$. Then a straightforward computation shows that $F$ has purely imaginary roots if and only if $r\in [-2,2]$. Moreover, $|F(0)|=C(r,\theta)\ne 0$ except for the two cases mentioned above, and in both exceptions, $F'(0)=-2$. The coefficient $b=b(r,\theta)$ is given by the real part of $\alpha$ in Remark \ref{entera}, just note that $e^{b X}=S^{b}$ and $e^{b Y}=R^{-b}$. Now apply Theorem \ref{ineq} to obtain the inequalities stated, observing that in both exceptions, 
$$
\alpha=\frac{d }{dz} \ln(F(z)/z)\left|_{z=0} \right.=0 .\qed
$$
\end{proof}

\begin{rem}\label{compl1}
If $r=0$ and $\theta\ne\pi$, then we obtain 
\begin{equation}\label{fujip}
\|e^{i\theta}STR^{-1}+S^{-1}TR\|_{\cal I}\ge \sqrt{2(1+\cos\theta)} \|T\|_{\cal I}.
\end{equation}
This inequality extends to 
\begin{equation}\label{fuji}
\|e^{i\theta}STR^{-1}+(S^*)^{-1}TR^*\|_{\cal I}\ge \sqrt{2(1+\cos\theta)} \|T\|_{\cal I}.
\end{equation}
where $S,R$ are just invertible. Indeed, put $S=|S|U$, $R=|R|V$ the polar decompositions of $S,R$. Then the left side of (\ref{fuji}) reads
$$
\|e^{i\theta}|S|U TV^*|R|^{-1}+|S|^{-1}UTV^*|R|\|_{\cal I}
$$
which by (\ref{fujip}) is greater or equal than
$$
\sqrt{2(1+\cos\theta)}\; \|UTV^*\|_{\cal I}=\sqrt{2(1+\cos\theta)}\|T\|_{\cal I}.
$$
\end{rem}

For $\theta=0,S=S^*=R$ and ${\cal I}=M=\bh$, inequality (\ref{fuji}) is known as the CPR inequality since its due to Corach-Porta and Recht \cite{cpr1}. Then Pedersen extended it for $S=R$ (not necessarily self-adjoint). Later Fujii-Fujii-Furuta and Nakamoto \cite{fujis} proved it for $R^*=R$, $S^*=S$, $R\ne S$. Then Kittaneh \cite{kita} proved it for general invertible $R,S\in \bh$, and unitarily invariant norms in $\bh$, that is
\begin{equation}\label{kita}
|||STR^{-1}+(S^*)^{-1}TR^*|||\ge 2 |||T|||.
\end{equation}

Kittaneh proves this inequality by showing that is in fact equivalent to the so called arithmetic-geometric-mean inequality \cite{agmi, agmi2}, that states
\begin{equation}\label{mean}
|||AA^*X+XBB^*|||\ge 2|||A^*XB|||
\end{equation}
for $A,B,X\in\bh$ and any unitarily invariant norm on $\bh$. Let us prove that Theorem \ref{grande} implies (\ref{mean}) in any normed ideal ${\cal I}$ with an unitarily invariant norm. Note that $X\to Y$ in $M$ implies $XT\to YT$ in ${\cal I}$ by property (\ref{ideal}) above. 

\begin{prop}\label{agmi}
Let $A,B\in M$ and $X\in {\cal I}$. Then
$$
\|AA^*X+XBB^*\|_{\cal I}\ge 2\|A^*XB\|_{\cal I}.
$$
\end{prop}
\begin{proof}
If we put $A=S$, $R=B$, $T=A^*XB$ in eq. (\ref{fujip}), we obtain this inequality for invertible $A,B\in M$. Assume now that $A,B$ are positive (not necessarily invertible). Let $A_{\varepsilon}=A+\varepsilon$, $B_{\varepsilon}=B+\varepsilon$. Then the inequality holds for $A_{\varepsilon},B_{\varepsilon}$ and any $X\in I$
$$
\|A_{\varepsilon}A_{\varepsilon}^*X+XB_{\varepsilon}B_{\varepsilon}^*\|_{\cal I}\ge 2\|A_{\varepsilon}^*XA_{\varepsilon}\|_{\cal I}.
$$
Letting $\varepsilon\to 0$ proves the inequality for positive $A,B\in M$. Now if $A=|A|U$ and $B=|B|V$ (polar decomposition), then
\begin{eqnarray}
\|AA^*X+XBB^*\|_{\cal I}&=& \| \,|A|^2X+X|B|^2\|_{\cal I}\ge 2\|\,|A|X|B|\,\|_{\cal I}\nonumber\\
\nonumber\\
&=& 2 \|U^*|A|X|B|V\|_{\cal I}=2\|A^*XB\|_{\cal I}.\nonumber
\end{eqnarray}
\end{proof}

\begin{rem}
If $\theta=0$, $r\ne -2$ then we have
$$
\|rT+STR^{-1}+S^{-1}TR\|_{\cal I}\ge \mid r+2\mid  \|T\|_{\cal I},
$$
which was proved for matrices by Zhan \cite[Cor. 7]{zhan}. It was obtained also by Bhatia and Parthasarathy \cite[Th. 5.1]{bhapar}, where they also show that the inequality is false for $r\notin (-2,2]$.
\end{rem}

\medskip

\begin{exmp}\label{compl2}
As a slight variation of Theorem \ref{grande},  consider $F(z)=e^z-e^{i\theta}$. Then (for $\theta=0$) we obtain 
$$
\|STR^{-1}-T\|_{\cal I}\ge \|S^{\frac12}XTR^{-\frac12}-S^{\frac12}TYR^{-\frac12}\|_{\cal I}
$$
for $S=e^X, R=e^{-Y}$ positive invertible and  $T\in {\cal I}$, and also
$$
\|STR^{-1}-e^{i\theta}T\|_{\cal I}\ge \sqrt{2(1-\cos\theta)}\|S^b TR^{-b}\|_{\cal I}.
$$
for $\theta\ne 2k\pi$, where $b=Re(\alpha)=\frac{1+\cos\theta}{2(1-\cos\theta)}$. In particular, for $\theta=\pi$ we get
$$
\|STR^{-1}+T\|_{\cal I}\ge 2\|T\|_{\cal I}.
$$
\end{exmp}

\medskip

\begin{rem}\label{analit}
The inequalities in this section can be rewritten if we note that $L_{F(X)}=F(L_X)$ for any  entire function $F$. For instance $$
\sinh(L_X-R_Y)=\sinh(L_X)\cosh(R_Y)-\sinh(R_Y)\cosh(L_X)
$$
and then (\ref{senoo}) reads
$$
\|\sinh(X)T\cosh(Y)-\cosh(X)T\sinh(Y)\|_{\cal I}\ge \|XT-TY\|_{\cal I}
$$
for $X,Y\in M_h$ and $T\in{\cal I}$. In particular
$$
\|\sinh(X)T\|_{\cal I}\ge \|XT\|_{\cal I}.
$$
Note also that equation (\ref{senoo}) can be easily generalized to non self-adjoint $R,S\in M$,
$$
\|STR^{-1}-(S^*)^{-1}TR^*\|_{\cal I}\ge 2\| XT-TY\|_{\cal I}
$$
where $e^X=|S|$ and $e^Y=|R|$ and now $S=U|S|$, $R=V|R|$ (right polar decomposition, i.e $|S|^2=S^*S$ and $|R|^2=R^*R$). This inequality can be found in \cite[Th. 4]{kosaki} and \cite[p. 223]{bhapar}; see also \cite[Th. 5]{kosaki} and Remark \ref{expo} below.
\end{rem}

\section{Second main inequality}\label{secmero}

In this section we compare quotients of entire functions by pairing their root sets. It should be noted that the interlacing condition of the roots $\cdots<w_k<z_k<w_{k+1}<z_{k+1}<\cdots$ that we require below for a pair of entire functions $F,G$ is related to the property that the quotient $F/G$ maps the upper half-plane into itself, thus the theorem below is related to a  theorem of L\"owner \cite{lowner} that states that a function $g:[0,+\infty)\to [0,+\infty)$ admits an analtytic extension that maps the upper-half plane into itself if and only if $g$ is operator monotone. As usual $M$ is a $C^*$-algebra and ${\cal I}$ a normed ideal in $M$ with an unitarily invariant norm.

\begin{thm}\label{comparo}
Let $X,Y\in M_h$. Let $F$, $G$ be as in Theorem \ref{ineq}, with $F(0)=G(0)=1$ and $F'(0)-G'(0)$ purely imaginary. Let $\{iz_k\}$ (resp. $\{iw_k\}$) be the roots of $F$ (resp. $G$). If for each positive (resp. negative) $z_k$ there is exactly one positive (resp. negative) $w_k$ with $z_k/w_k>1$, then the quotient
$$
H(L_X+R_Y)=F(L_X+R_Y)G^{-1}(L_X+R_Y)
$$
is a contraction of ${\cal B}({\cal I})$.
\end{thm}
\begin{proof}
We can write  $e^{-[F'(0)-G'(0)]z}H(z)$ as an infinite product
$$
\prod_{\{k:z_k,w_k>0\}}(1-\frac{iz}{z_k})(1-\frac{iz}{w_k})^{-1}e^{\frac{iz}{z_k}-\frac{iz}{w_k}}\,\prod_{\{k:z_k,w_k<0\}}(1-\frac{iz}{z_k})(1-\frac{iz}{w_k})^{-1}e^{\frac{iz}{z_k}-\frac{iz}{w_k}}.
$$
Now an elementary computation shows that
$$
(1-\frac{iz}{z_k})(1-\frac{iz}{w_k})^{-1}=\left[\frac{z_k}{w_k}+ (1-\frac{z_k}{w_k})(1+\frac{iz}{z_k})^{-1}\right]^{-1}.
$$
Each of these factors (when evaluated in $L_X+R_Y$) is a contraction by Lemma \ref{divido}, with $a=z_k$, $b=w_k$ and 
$$
A=(1+\frac{i}{z_k}(L_X+R_Y))^{-1}.\qed
$$
\end{proof}

\begin{cor}
Let $F,G$ be as in Theorem \ref{ineq}, with $F(0)=G(0)=1$, let $A=L_X+R_Y$ with $X,Y\in M_h$. Then
$$
\|F(A)T\|_{\cal I}\le\|G(A)T\|_{\cal I}
$$
for any $T\in {\cal I}$. 
\end{cor}

\begin{cor}\label{comparo2}
Let $F$ be as in Theorem \ref{ineq}, with $F(0)=1$, let $A=L_X+R_Y$ with $X,Y\in M_h$. Then 
$$
\|F(sA)F(A)^{-1}T\|_{\cal I}\le \|e^{F'(0)(s-1)A} T \|_{\cal I},
$$
for any $T\in {\cal I}$, for any $s\in (0,1)$.
\end{cor}
\begin{proof}
Consider the auxiliary functions 
$$
\bar{F}(z)=e^{-F'(0)z}F(z), \quad G_s(z)=\bar{F}(sz).
$$
Then $\bar{F}'(0)=0$ and $G_s'(z)=0$, and by the previous theorem, $(\bar{F}^{-1}G_s)(A)$ is a contraction for each $A$. 
\end{proof}

\begin{prop}\label{divi}
Let $X,Y\in M_h$. Let $A=L_X+R_Y$ and $s\in [0,1]$. Then each of the maps listed below is a contraction of ${\cal B}({\cal I})$:
\begin{enumerate}
\item $\displaystyle\frac{\sinh(sA)}{s\sinh(A)}$. The case $s=0$ should be understood as  $\displaystyle\frac{A}{\sinh(A)}$.
\smallskip
\item $\displaystyle\frac{\cosh(sA)}{\cosh(A)}$.
\smallskip
\item $\displaystyle\frac{\sinh(sA)}{sA\cosh(rA)}$ if $0\le s\le r$. In particular $\displaystyle\frac{\tanh(A)}{A}$ is a contraction.
\smallskip
\item $\displaystyle\frac{sA\cosh(rA)}{\sinh(sA)}$ if $0\le 2r\le s$.
\end{enumerate}
\end{prop}
\begin{proof}
Immediate from the two previous results, since both $\sinh(z)/z$ and $\cosh(z)$ have zero derivative at $z=0$.
\end{proof}

The maps of the above proposition have been studied in several papers by many different authors, let just mention a few references such as \cite{bhatiakosaki, bhapar, kosakih, kosaki}.

\begin{rem}\label{maci}
Let $t\in [0,1]$, $A,B\in M$ positive, $T\in {\cal I}$. Then
$$
\|A^{1-t}TB^{t}+A^{t}TB^{1-t}\|_{\cal I}  \le  \|AT+TB\|_{\cal I}
$$
and
$$
\|A^{1-t}TB^{t}-A^{t}TB^{1-t}\|_{\cal I} \le  (2t-1) \|AT-TB\|_{\cal I}.
$$
These inequalities for $M={\cal I}=\bh$ are due to Heinz \cite{heinz}, later Bhatia and Davies \cite{bhatiadaviesuno} generalized them to unitarily invariant norms in $\bh$. They follow easily from the previous proposition if we let $A=e^X$, $B=e^Y$ and $D=L_X-R_Y$. Then
\begin{eqnarray}\label{cosh}
A^{1-t}TB^{t}+A^{t}TB^{1-t}=\cosh((2t-1)D)(T)\\
\nonumber\\
A^{1-t}TB^{t}-A^{t}TB^{1-t}=\sinh((2t-1)D)(T).\nonumber
\end{eqnarray}
\end{rem}

\subsection{The exponential metric increasing property and operator means}

Assume that $t\in [\frac14,\frac34]$, let $A,B\in\bh$ be positive. If $|||\cdot|||$ is any unitarily invariant norm in $\bh$, then it is possible to compare the following means:
\begin{eqnarray}\label{refinada}
|||A^{\frac12}TB^{\frac12}||| &\le &\frac12 |||A^{t}TB^{1-t}+A^{1-t}TB^t|||\\
&\le & |||\int_0^1 A^{1-s}TB^{s}ds||| \le \frac12 |||AT+TB|||.\nonumber
\end{eqnarray}
This chain of inequalities was proved (in several steps in \cite{bhatiadavies,bhatiakosaki,kosakih} and with different techniques) by Bhatia, Davis, Hiai and Kosaki. In fact, in \cite[Th. 1]{D}, the optimality of the interval $[\frac14,\frac34]$ for the second term was shown by Drissi. In this section we indicate a proof by specialization, and comment some related inequalities.

\begin{rem}\label{expo}
This inequality relating several means is related to the EMI property (\textit{exponential metric increasing property}) since, if we put $T=A^{-1/2}YB^{-1/2}$ and $A=B=e^X$, we get
$$
|||e^{-X/2}\, (\int_0^1 e^{(1-t)X}Ye^{tX}dt)\, e^{-X/2}|||=|||\int_0^1 e^{(1/2-t)X}Ye^{(t-1/2)X}dt|||\ge |||Y|||,
$$
and since $\int_0^1 e^{(1-t)X}Te^{tX}dt=d\exp_X(Y)$, then we obtain
$$
|||e^{-X/2}d \exp_X(Y)e^{-X/2}|||\ge |||Y|||.
$$
Here $\exp:M\to M$ is the usual exponential map $\exp(X)=e^X$ and $d \exp_X$ denotes the differential of $\exp$ at $X$. The identity
\begin{equation}\label{difexpo}
\int_0^1 e^{(1-t)X}Y e^{tX}dt=d\exp_X(Y)
\end{equation}
is elementary. Indeed, put $f(t)=e^{(1-t)X}$ and $g(t)=e^{t(X+Y)}$. Then integrating the product by parts in $[0,1]$ yields
$$
e^{X+Y}-e^X=\int_0^1 e^{(1-t)X}Ye^{t(Y+X)}\,dt.
$$
Replacing $Y$ by $sY$ and letting $s\to 0$ gives (\ref{difexpo}). This last computations are formal and can be carried out in any normed ideal ${\cal I}$ with a unitarily invariant norm.
\end{rem}

For $X,Y\in M$ and $T\in{\cal I}$ we have
\begin{eqnarray}
\int_0^1 e^{(t-\frac12)X}Te^{(\frac12-t)Y} \,dt &=&\sum_{n\ge 0} \frac{1}{n!}\int_0^1 (t-\frac12)^n\,dt\,(L_X-R_Y)^n(T)\nonumber\\
&=&\sum_{k\ge 0} \frac{1}{2^{2k}(2k+1)!}  (L_X-R_Y)^{2k}(T).\nonumber
\end{eqnarray}
Hence
\begin{equation}\label{integral}
\int_0^1 e^{(t-\frac12)X}Te^{(\frac12-t)Y} \,dt=\frac{\sinh (z/2)}{z/2}\left|_{L_X-R_Y} (T)\right.= F(L_X-R_Y)(T), 
\end{equation}
where $F(z)=\frac{\sinh (z/2)}{z/2}$. In particular,
\begin{equation}\label{difexpon}
e^{-\frac {X}{2}}d\exp_X (Y) e^{-\frac {X}{2}} =F(L_X-R_X)(Y) =F(ad(X))(Y).
\end{equation}

We obtain a generalization of (\ref{refinada}) using the above remarks and Proposition \ref{divi}. For instance from
\begin{eqnarray}
\|S\|_{\cal I} &\le &\|\cosh((t-1/2)(L_X-R_Y))S\|_{\cal I}\nonumber\\
\nonumber\\
&\le &\left\|\frac{\sinh(s(L_X-R_Y))}{s(L_X-R_Y)}S\right\|_{\cal I}\le \|\cosh(r(L_X-R_Y))S\|_{\cal I},\nonumber
\end{eqnarray}
valid for $t\in [\frac12,1]$ and $r\ge s\ge 2t-1\ge 0$ by Proposition \ref{divi}, we obtain (\ref{refinada}) if we put $r=s=1/2$ and let $t\in [\frac12,\frac34]$, put $A=e^X$, $B=e^Y$ and $S=A^{\frac12}TB^{\frac12}$ and use equations (\ref{integral}) and (\ref{cosh}). The usual trick with $A_{\varepsilon}=A+\varepsilon1$ and $B_{\varepsilon}=B+\varepsilon1$ gives the inequality for positive (not necessarilly invertible) operators, and then by symmetry it extends for $t\in [\frac14,\frac34]$.

\medskip

Since it will be useful later, let us state as a corollary the exponential metric increasing property (or EMI for short).

\begin{cor}\label{expande}
Let $X,Y\in M_h\cap {\cal I}$. Then
$$
\|e^{-\frac{X}{2}}d \exp_X(Y)e^{-\frac{X}{2}}\|_{\cal I}\ge \|Y\|_{\cal I}.
$$
\end{cor}

\medskip

\begin{rem}\label{seno}If $iX=H,iY=K\in iM_h$ are skew-adjoint and $T\in {\cal I}$, then
\begin{eqnarray}
\|\sin(L_H-R_K)T\|_{\cal I} & = & \left\|\sinh(L_{iX}-R_{iY})T\right\|_{\cal I}\nonumber \\
\nonumber\\
 &= &\left\|\frac{\sinh(L_{iX}-R_{iY})}{L_{iX}-R_{iY}}  (HT-TK)\right\|_{\cal I}\nonumber \\
 &= &\|\int_0^1 e^{(1-2t)iX} (HT-TK) e^{(2t-1) iY}dt\|_{\cal I}\nonumber\\
 &\le& \|HT-TK\|_{\cal I}\nonumber
\end{eqnarray}
by equation (\ref{integral}), since $e^{iX}$ and $e^{iY}$ are unitary operators. Hence
$$
\|\sin(H)T\cos(K)-\cos(H)T\sin(K)\|_{\cal I}\le \|HT-TK\|_{\cal I}
$$
and in particular we obtain an inequality that can be found in \cite[Th. 4]{kosaki}.
$$
\|\sin(H)T\|_{\cal I}\le \|HT\|_{\cal I}.
$$
See also the related Remark \ref{analit}.
\end{rem}

\begin{exmp}
Let $X,Y\in M_h$ and $T\in {\cal I}$. Let $\lambda_n$ be the root of $\tan(x)=x$ in $(n\pi, (n+\frac12)\pi)$. Then it is not hard to see \cite[p. 233]{bere} that
$$
z\cosh(z)-\sinh(z)=\frac{z^3}{3}\prod_{n\ge 1} \left(1+\frac{z^2}{\lambda_n}  \right).
$$
Dividing by $z$ and using eq. (\ref{integral}) we obtain, for $S=e^X$ and $R=e^Y$
$$
\|STR^{-1}+RTS^{-1}-\int_0^1S^{2t-1} TR^{1-2t}\,dt\|_{\cal I} \ge \frac23 \|X^2T+TY^2-2XTY\|_{\cal I}.
$$
\end{exmp}

\subsection{Lipschitz maps of operators}

\begin{thm}\label{uncoso}
Let $T\in {\cal I}$. Let $R,S\in M$ be invertible, $|S|=e^X,|R|=e^Y$. Then
$$
\|STR^{-1}-(S^*)^{-1}TR^*\|_{\cal I}\le \|\tanh(L_X-R_Y)\|_{{\cal B}({\cal I})} \|STR^{-1}+(S^*)^{-1}TR^*\|_{\cal I}.
$$
\end{thm}
\begin{proof}
We may assume that $R$ and $S$ are positive and invertible, $S=e^X,R=e^{Y}$. Since $\cosh(0)=1$, then $D=\cosh(L_X-R_Y)$ is invertible in ${\cal B}({\cal I})$ and then
\begin{eqnarray}
\|\sinh(L_X-R_Y)T\|_{\cal I}  & = &\|\sinh(L_X-R_Y)D^{-1} D(T)\|_{\cal I} \nonumber\\
\nonumber\\
& \le & \|\tanh(L_X-R_Y)\|_{  {\cal B}({\cal I})  } \|\cosh(L_X-R_Y)(T)\|_{\cal I}.\nonumber
\end{eqnarray}
\end{proof}

\begin{cor}\label{otrocoso}
Let $T\in {\cal I}$. Let $R,S\in M$ be invertible, $|S|=e^X, |R|=e^Y$. Then
$$
\|STR^{-1}-(S^*)^{-1}TR^*\|_{\cal I}\le \|L_X-R_Y\|_{{\cal B}({\cal I})}\|STR^{-1}+(S^*)^{-1}TR^*\|_{\cal I}.
$$
\end{cor}

\medskip

\begin{rem}
In \cite{seddik2} the author proves
\begin{equation}\label{sed}
|||STS^{-1}-S^{-1}TS|||\le (\|S\|\,\|S^{-1}\|-1)|||STS^{-1}+S^{-1}TS|||
\end{equation}
for $S\in \bh$ invertible self-adjoint, and $T$ in a normed ideal in $\bh$ with unitarily invariant norm $|||\cdot |||$. Recall $ad(X)=L_X-R_X$ is called the \textit{adjoint representation}. We claim that for $X\in M_h$
\begin{eqnarray}
\|ad(X)\|_{{\cal B}({\cal I})} & \le & \|ad(X)\|_{{\cal B}({\cal B}(\cal H))}=\lambda_{max}(X)-\lambda_{min}(X)\nonumber\\
&<& e^{\lambda_{max}(X)-\lambda_{min}(X)}-1= \|e^X\|\,\|e^{-X}\|-1,\nonumber
\end{eqnarray}
viz. that the previous corollary improves eq. (\ref{sed}). Here $\lambda_{max}(X)$ and $\lambda_{min}(X)$ stand for the largest and smallest spectral values of $X$ in $M$. The substantial part is the first inequality, to prove it note that $ad(X)$ is an Hermitian element of ${\cal B}({\cal I})$, so its norm can be computed with the spectral radius formula \cite[28.3]{levin}, and 
$$
\sigma_{{\cal B}({\cal I})} (ad(X))\subset \sigma(X)-\sigma(X)
$$
by Remark \ref{espectro}.
\end{rem}

\begin{rem}
Since $\tanh$ maps $\mathbb R$ onto $(-1,1)$, if we put 
$$
\Theta(X,Y)(T)=\tanh(L_X-R_Y)(T),
$$
then
$$
\sigma_{{\cal B}({\cal I})}(\Theta(X,Y))\subset (-1,1).
$$
From Theorem \ref{uncoso}, we know that
$$
\| STR^{-1}-(S^*)^{-1}TR^*\|_{\cal I} \le \, \|\Theta(X,Y)\|_{{\cal B}({\cal I})}\, \| STR^{-1}+(S^*)^{-1}TR^*\|_{\cal I}
$$
for invertible $S,R\in M$ and $T\in {\cal I}$. If $M$ is a semi-finite factor with trace $\tau$, and ${\cal I}=L^p(M,\tau)$, then for $p=2$ the space ${\cal I}$ is a Hilbert space. By Remark \ref{eledos}, the norm of $\Theta(X,Y)$ can be computed via the spectral radius. Then 
$$
\| \Theta(X,Y)\|_{ {\cal B}(L^2(M,\tau)) }\le 1
$$
for any $X,Y\in M_h$ (i.e. $\Theta$ is uniformly bounded for any $X,Y\in M_h$). 
\end{rem}

If ${\cal I}$ is not a Hilbert space, the norm of this operator can be larger than its spectral radius. However, if we put put $A=S$, $R=B$, $T=A^*XB$ (as in Proposition \ref{agmi}) and then replace $|A|$ and $|B|$ by its square roots, we obtain 
$$\| |A|X-X|B|\|_{\cal I} \le \, \|\Theta(1/2\ln|A|,1/2\ln|B|)\|_{{\cal B}({\cal I})}\, \| |A|X+X|B|\|_{\cal I},
$$
and it was proved in \cite{davies} by Davies that for $p\in (1,+\infty)$
$$
\| |A|X-X|B|\|_{p} \le \, \gamma_p \| |A|X+X|B|\|_{p},
$$
for invertible $R,S\in \bh$ and $T\in \bp=L^p(\bh,\tau)$ (the ideal of compact $p$-Schatten operators). The constant $\gamma_p$ does not depend on the involved operators, but depends on $p$, and it is equal to one when $p=2$. The proof of Davies uses a nontrivial result of Macaev \cite{macaev} for linear transformators on compact operators which can be found in p.121 of Gohberg and Krein's book \cite{gk}. See \cite[2.1]{kita} for further discussion on the subject.

The continuity of the absolute value map in Banach spaces associated with general von Neumann algebras was proved later by Dodds-Dodds-Pagter and Sukochev \cite{dd}, where a full list of references on the subject can be found. In other direction, the inequality
$$
\|STS^{-1}-S^{-1}TS\|_{\cal I}\le \|ad X\|\|STS^{-1}+S^{-1}TS\|_{\cal I}.
$$
has been interpolated  by Conde \cite{condeineq} to certain classes of spaces related to the Finsler manifold of positive operators, for symmetrically normed ideals of $\bh$.

\subsection{The L\"owner-Heinz inequality}

The inequality
$$
\|A^tB^t\|\le \|AB\|^{t},\quad \;t\in [0,1],
$$
valid for positive invertible operators $A,B\in\bh$, is occasionally called the C\"ordes inequality in the literature, and it is also known as the L\"owner-Heinz inequality since it is equivalent to the fact that $t$-power ($t\in [0,1]$) is operator monotone \cite{furuta}. For the $p$-norms of $\bh$ ($p>0$) it is stated as
$$
Tr((B^{\frac12} A B^{\frac12})^{rp})\le Tr((B^{\frac{r}{2}} A^r B^{\frac{r}{2}})^p),\quad r\ge 1,
$$
an inequality due to Araki \cite{araki} (here $Tr$ denotes the usual infinite trace of $\bh$). As it is, it was generalized to the noncommutative $L^p(M,\tau)$-spaces of a semi-finite von Neumann algebra $M$ by Kosaki in \cite{kosa2}.

In the uniform norm of $\bh$ it has an equivalent expression 
$$
\|\ln(A^{-\frac{t}{2}}B^{t}A^{-\frac{t}{2}})\|\le t\|\ln(A^{-\frac{1}{2}}B A^{-\frac{1}{2}})\|,
$$
and stated in this form, it establishes the convexity of the geodesic distance in the Finsler manifold of positive invertible operators \cite{corachproc}, when they are regarded as an homogeneous space of the full group of invertible operators by the action 
$$
A\mapsto GAG^*, \quad A\in \bh^+,\;G\in \bh^{\times}.
$$
Here $\bh^{\times}$ denotes the group of invertible elements in $\bh$. In this section we prove the   Araki-C\"ordes inequality for unitarily invariant norms in a $C^*$-algebra $M$.

\begin{rem}
When ${\cal I}_0={\cal I}\cap M$ is an ideal with an unitarily invariant norm in a $C^*$-algebra $M$, there is a natural bijection 
$$
{\cal I}\cap M_h \longleftrightarrow {\cal I}_0^+=\{ 1 + T>0,\;T\in {\cal I}_0\}.
$$
between bounded self-adjoint elements of ${\cal I}_0$ and the (unitized) positive invertible elements of ${\cal I}_0$, given by the usual exponential map of $M$. In fact, if $T\in M_h\cap {\cal I}$, then clearly
$$
e^T=1+T+\frac12 T^2+\cdots=1+T(1+\frac12 T+\cdots)\in {\cal I}_0^+.
$$
On the other hand, if $1+T\in {\cal I}_0^+\subset M$, it has (being positive and invertible) a unique real analytic logarithm $X\in M_h$ (given for instance by the Cauchy functional calculus). Then if we consider $F(z)=z^{-1}(e^z-1)$, since $e^X=1+T$ we have
$$
1+T=e^X=1+XF(X),
$$
hence $T=XF(X)$. But for self-adjoint $X\in M_h$, the element $F(X)$ is invertible in $M$ hence $X=F(X)^{-1}T$ which proves that $X\in {\cal I}\cap M_h$.
\end{rem}

\begin{thm}
Let $X,Y\in {\cal I}\cap M_h$ with  $({\cal I},\|\cdot\|_{\cal I})$ a complex normed ideal in $M$ with an unitarily invariant norm, and let $t\in [0,1]$. Then
$$
\|\ln(e^{-\frac{t}{2}X}e^{tY}e^{-\frac{t}{2}X})\|_{\cal I}\le \; t\, \|\ln(e^{-\frac{X}{2}}e^{Y}e^{-\frac{X}{2}})\|_{\cal I}.
$$
\end{thm}
\begin{proof}
Let 
$
X_t=\ln(e^{-\frac{t}{2}X}e^{tY} e^{-\frac{t}{2}X}),\;$ then we have to show that $\| X_t \|_{\cal I} \le t\; \|X_1\|$. Consider 
$$
g(s)=e^{\frac{X}{2}}e^{sX_1} e^{\frac{X}{2}},
$$
which is positive invertible for any $s\in \mathbb R$, and consider the auxiliary function 
$$
\beta_t(s)=e^{-\frac{t}{2}X}g(s)^te^{-\frac{t}{2}X},
$$
which is also positive invertible for any $s,t\in\mathbb R$. In this proof the dot indicates the derivative with respect to the $s$ variable, and this variable is omitted in the computations. 

First note that, if $U_t=\beta_t^{-\frac{1}{2}}e^{-\frac{t}{2}X}g^{\frac{t}{2}}$, then $U_t$ is a unitary operator, hence
$$
\| g^{-\frac{t}{2}} (g^t)^{\cdot}g^{-\frac{t}{2}} \|_{\cal I}= \| U_t g^{-\frac{t}{2}} (g^t)^{\cdot}g^{-\frac{t}{2}} U_t^*\|_{\cal I}=\| \beta^{-\frac{1}{2}}_t\dot{\beta}_t \beta^{-\frac{1}{2}}_t \|_{\cal I}.
$$
Since $\beta_t > 0$, $\dot{\beta_t}= d\exp_{\ln\beta_t}(\frac{d}{ds}\ln\beta_t)$, hence 
$$
\|\beta_t^{-\frac{1}{2}}\dot{\beta_t} \,\beta_t^{-\frac{1}{2}} \|_{\cal I}= \| e^{-\frac{\ln\beta_t}{2}} \dot{\beta_t}\,e^{-\frac{\ln\beta_t}{2}}\|_{\cal I}\ge \|\frac{d}{ds}\ln\beta_t\|_{\cal I}
$$
by Corollary \ref{expande}. 

From the facts $\beta_t(0)=1$ and $X_t=\ln\beta_t(1)$ (for any $t\in\mathbb R$) we obtain
\begin{eqnarray}
\|X_t\|_{\cal I} &=& \| \ln\beta_t(1)-\ln\beta_t(0)\|_{\cal I}\le \int _0^1 \|\frac{d}{ds}\ln\beta_t\|_{\cal I} \;ds \nonumber\\
&\le &\int_0^1 \|\beta_t^{-\frac{1}{2}}\dot{\beta_t} \beta_t^{-\frac{1}{2}} \|_{\cal I} \;ds= \int_0^1 \|g^{-\frac{t}{2}} (g^t)^{\cdot}g^{-\frac{t}{2}}\|_{\cal I} \; ds \,\nonumber.
\end{eqnarray}
On the other hand 
\begin{eqnarray}
g^{-\frac{t}{2}} (g^t)^{\cdot}g^{-\frac{t}{2}} &= &t\; g^{-\frac t2} d\exp_{t\ln g}(d\ln _g(\dot{g})) g^{-\frac t2}\nonumber\\
\nonumber\\
&=& t\; F(t \, ad(\ln g)) F(ad(\ln g))^{-1} (g^{-\frac{1}{2}} \dot g g^{-\frac{1}{2}})\nonumber
\end{eqnarray}
by the chain rule and equation (\ref{difexpon}) in Remark \ref{expo}, where $F(z)=\frac{\sinh (z/2)}{z/2}$. Then, if $t\in (0,1)$, 
$$
\|g^{-\frac{t}{2}} (g^t)^{\cdot}g^{-\frac{t}{2}}\|_{\cal I} \le t \; \|g^{-\frac{1}{2}} \dot g g^{-\frac{1}{2}}\|_{\cal I}
$$
by Corollary \ref{comparo2}. Finally, note that $V_s=g(s)^{-\frac 12} e^{\frac{X}{2}} e^{\frac s2 X_1}$ is a unitary operator, hence
$$
\| g^{-\frac{1}{2}} \dot g g^{-\frac{1}{2}}\|_{\cal I} = \| V_s X_1 V_s^* \|_{\cal I} =\| X_1 \|_{\cal I}.
$$
\end{proof}

\begin{rem}
In the above proof, we in fact proved the technical inequality
$$
\|\ln(e^{-\frac{t}{2}X}e^{tY}e^{-\frac{t}{2}X})\|_{\cal I}\le \int_0^1\|g^{-\frac{t}{2}} (g^t)^{\cdot}g^{-\frac{t}{2}}\|_{\cal I}\,ds \le \; t\, \|\ln(e^{-\frac{X}{2}}e^{Y}e^{-\frac{X}{2}})\|_{\cal I},
$$
with $g(s)=e^{\frac{X}{2}}(e^{-\frac{X}{2}}e^{Y}e^{-\frac{X}{2}} )^s e^{\frac{X}{2}}$.
\end{rem}

\section{Further generalizations}\label{secfinal}

The hypothesis on the order of the entire function $F$ of Theorem \ref{in} can be relaxed in order to obtain inequalities for functions of any finite order. Let us indicate here two other possible generalizations of the main results of this paper.

\begin{enumerate}
\item If ${\cal A}$ is an Hermitian Banach algebra (in the sense that it has an isometric involution and each element of the form $a^*a$ has nonnegative spectrum, or equivalently that hermitian elements have real spectrum  \cite[Section 11.4]{palmer}), then Theorem \ref{ineq} and Theorem \ref{comparo} can be rewritten in that context with no significant modification.

\medskip

\item If we consider the ideals ${\cal I}^{\tau,p}$ of a semi-finite von Neumann algebra with a semi-finite trace $\tau$ (consisting of the bounded elements of the noncommutative $L^p$ spaces), then the role of the ideal in the main inequalities can be reversed in the following sense: instead of considering $X,Y\in M_h$ (bounded) and $T\in L^p(M,\tau)$, consider $X,Y\in L^p(M,\tau)_h$ and $T\in {\cal I}^{\tau,p}$. In order to make sense out of expressions such as $F(L_X+R_Y)T$ or $e^XTe^{-Y}=e^{L_X}e^{-R_Y}T$, we use analytic vectors \cite{nelsonlitic}. We will focus on left multiplication, the right multiplication operator can be treated in the same fashion. Note that $L_X$ is a densely defined operator in $L^p(M,\tau)$ (recall that $\|XT\|_p\le \|X\|_p \|T\|_{\infty}$). Assume that $L_X$ admits a dense set of analytic vectors. Then $F(L_X)$ is a (possibly unbounded) operator of $L^p(M,\tau)$, defined on a dense set $V_X\subset L^p(M,\tau)$. Moreover, since $\|e^{i L_X}T\|_p=\|T\|_p$ for any $T\in V_X\subset L^p(M,\tau)$,  by the theory of unbounded dissipative operators the operator $1\pm s i L_X$ is expansive (and possibly unbounded) in $L^p(M,\tau)$, with bounded inverse (for any  $s\in\mathbb R$). The rest of the proof of Theorem \ref{ineq} follows replacing ${\cal I}$ with $V_X$. The same remarks hold for Theorem \ref{comparo}. Note that for $p=2$, the operator $L_X$ is self-adjoint in ${\cal H}=L^2(M,\tau)$. Hence $L_X$ always admits a dense set of analytic vectors: for consider a resolution of the identity ${\cal E}$ of $L_X$, then the family
$$
V_X=\{{\cal E}(A)T:A\subset\mathbb R\mbox{  is a bounded Borel set, } T\in L^2(M,\tau)\} 
$$
is a dense set of analytic vectors for $L_X$ by the spectral theorem.
\end{enumerate}

\bigskip

\noindent
Gabriel Larotonda\\
Instituto de Ciencias \\
Universidad Nacional de Gral. Sarmiento \\
J. M. Gutierrez 1150 \\
(1613) Los Polvorines \\
Argentina  \\
e-mail: glaroton@ungs.edu.ar

\end{document}